\documentclass[12pt]{article}
\title{The $T_{4}$ and $G_{4}$ constructions of Costas arrays}
\author{Tim Trudgian\footnote{Supported by Australian Research Council DECRA Grant DE120100173.}\\
Mathematical Sciences Institute\\ The Australian National University,
 ACT 0200, Australia\\ timothy.trudgian@anu.edu.au\\ \\ and \\ \\
 Qiang Wang\footnote{Supported by NSERC of Canada.}\\ School of Mathematics and Statistics\\ Carleton University\\ Ottawa, Ontario, K1S 5B6, Canada\\ wang@math.carleton.ca
}
\usepackage{url}
\usepackage{amsthm}
\usepackage{comment}
\usepackage{amsmath}
\usepackage{fullpage}
\usepackage{amssymb}

\usepackage{booktabs}
\newtheorem{thm}{Theorem}

\newtheorem{lem}{Lemma}
\begin{document}

\maketitle

\begin{abstract}
\noindent We examine two particular constructions of Costas arrays known as the Taylor variant of the Lempel construction, or the $T_{4}$ construction, and the variant of the Golomb construction, or the $G_{4}$ construction. We connect these constructions with the concept of Fibonacci primitive roots, and show that under the Extended Riemann Hypothesis the $T_{4}$ and $G_{4}$ constructions are valid infinitely often.
\end{abstract}

\section{Introduction}
A Costas array is an $N\times N$ array of dots with the properties that one dot appears in each row and column, and that no two of the $N(N-1)/2$ line segments connecting dots have the same slope and length.  
It is clear that a permutation $f$ of $\{1,2,\ldots,N\}$, from
the columns to the rows (i.e.\ to each column $x$ we assign exactly one row $f(x)$), gives a Costas array if and only if for $x
\neq y$ and $k \neq 0$ such that $1 \leq x, y, x+k, y+k \leq N$,
then $f(x+k)-f(x) \neq f(y+k)-f(y)$. 

Costas arrays were first considered by Costas \cite{costas1975} as
permutation matrices with ambiguity functions taking only the
values 0 and (possibly) 1, applied to the processing of radar and
sonar signals. The use of Costas arrays in radar is summarized in
\cite[\S 5.2]{levanon}. 
Costas arrays are also used in the design of optical orthogonal
codes for code division multiple access (CDMA) networks
\cite{maric1995}, and in the construction of low-density
parity-check (LDPC) codes \cite{chae2004}. 

Let us briefly recall some  known constructions  on
Costas arrays.  One can find more details in the survey papers
of Golomb and Taylor \cite{MR674209,golomb1984}, Drakakis \cite{drakakis}, Golomb and Gong \cite{GolombGong}.
In the following, $p$ is taken to be a prime and $q$ a prime power.
 The known general
constructions for $N \times N$ Costas arrays are the Welch
construction for $N=p-1$ and $N=p-2$, the Lempel construction for
$N=q-2$, and the Golomb construction for $N=q-2$, $N=q-3$. Moreover,
if $q=2^k$, $k\geq 3$, the Golomb construction works for
$N=q-4$.
 The validity of the Welch and Lempel constructions is proved by
Golomb in \cite{MR749508}. The Golomb constructions for $N=q-3$ and
$N=2^k-4$ depend on the existence of (not necessarily distinct)
primitive elements $\alpha$ and $\beta$ in $\mathbb{F}_q$ such that
$\alpha+\beta=1$.   
The existence of primitive elements $\alpha$ and $\beta$ in $\mathbb{F}_q$ such that
$\alpha+\beta=1$  was proved by Moreno
and Sotero in \cite{moreno1990}. (Cohen and Mullen give a proof with
less computational checking in \cite{MR1209243}; more recently, Cohen, Oliveira e Silva, and Trudgian proved \cite{COT1} that, for all $q>61$, every non-zero element in $\mathbb{F}_{q}$ can be written as a linear combination of two primitive roots of $\mathbb{F}_{q}$.)

Among these algebraic constructions over finite fields,  there are the $T_4$ variant of the Lempel
construction for $N=q-4$ when there is a primitive element $\alpha$
in $\mathbb{F}_q$ such that $\alpha^2+\alpha=1$,  
  and the $G_4$ variant of the Golomb construction for $N=q-4$ when there are two primitive elements $\alpha$ and $\beta$ such that $\alpha + \beta = 1$ and $\alpha^{2} + \beta^{-1} = 1$. 
Through the study of primitive elements of finite fields,  Golomb proved in   
\cite{golomb1992} that $q$ must be either 4, 5 or 9, or a prime $p
\equiv \pm 1 \pmod{10}$ in order for the $T_4$ construction to apply. 
Note that this is a necessary but not sufficient condition (for example $p=29$).   In the same paper, 
Golomb also proved that the values of $q$ such that the $G_4$ construction
occurs are precisely $q = 4, 5, 9$, and those
primes $p$ for which the $T_4$ construction occurs and which satisfy either
$p \equiv 1 \pmod{20}$  or $p \equiv 9 \pmod{20}$.

In this paper, we connect the $T_4$ and $G_{4}$ constructions with the concept of Fibonacci primitive roots. We show, in Theorems \ref{t1} and \ref{t2}, that under the Extended Riemann Hypothesis (ERH) there are infinitely many primes such that $T_4$ and $G_4$ can apply. We conclude with some observations and questions about trinomials of primitive roots.

\section{Fibonacci primitive roots}\label{section:FPR}

The $T_{4}$ construction requires a primitive root $\alpha$ such that
\begin{equation}\label{dog}
\alpha^{2} + \alpha = 1.
\end{equation}
To investigate the nature of solutions to (\ref{dog}) we recall the notion of a \textit{Fibonacci primitive root}, or \textit{FPR}. 
We say that $g$ is a FPR modulo $p$ if $g^{2} \equiv g + 1 \pmod p$. Shanks and Taylor \cite{ShanksTaylor} proved a similar statement to that which we give below.

\begin{lem}\label{onlyLem}
If $g$ is a FPR modulo $p$, then $g-1$ is a primitive root modulo $p$ that satisfies (\ref{dog}), and vice versa.
\end{lem}
\begin{proof}
It is clear that $g$ satisfies $g^{2} \equiv g + 1 \pmod p$ if and only if $g-1$ satisfies (\ref{dog}): all that remains is to check that $g$ and $g-1$ are primitive. Suppose first that $g$ is a FPR modulo $p$. Then, since $g(g-1) \equiv 1\equiv g^{p-1}$, we have
\begin{equation*}\label{cat}
(g-1)^{n} \equiv g^{p-n-1} \pmod p, 
\end{equation*}
Note that, as $n$ increases from $1$ to $p-1$, $g^{p-n-1}$ generates $\mathbb{F}_{p}$, since $g$ is primitive. Hence $g-1$ is a primitive root modulo $p$. The converse is similarly proved.
\end{proof}
Let $F(x)$ denote the number of primes $p\leq x$ that have at least one FPR. Shanks \cite{ShanksFPR} conjectured that under ERH, $F(x) \sim C \pi(x)$, where $\pi(x)$ is the prime counting function, and where $C \approx 0.2657\ldots $. Lenstra \cite{LenstraArtin} proved Shanks' conjecture; a proof also appears in Sander \cite{Sander}. 
We therefore have
\begin{thm}\label{t1}
Let $T(x)$ be the number of primes $p\leq x$ for which $p$ satisfies the $T_{4}$ construction. Then, under the Extended Riemann Hypothesis
\begin{equation*}\label{t1:e1}
T(x) \sim \frac{27}{38} \pi(x) \prod_{p=2}^{\infty} \left( 1- \frac{1}{p(p-1)}\right) \sim (0.2657\ldots) \pi(x).
\end{equation*}
\end{thm}
Unconditionally, it seems difficult to show that there are infinitely many primes that have a FPR. Phong \cite{Phong} has proved some results about a slightly more general class of primitive roots. For our purposes, \cite[Cor.\ 3]{Phong} implies that if $p\equiv 1, 9 \pmod{10}$ such that $\frac{1}{2}(p-1)$ is prime then there exists (exactly) one FPR modulo $p$. This does not appear, at least to the authors, to make the problem any easier!

We turn now to the $G_{4}$ construction, which requires two primitive roots $\alpha, \beta$ such that
\begin{equation*}\label{dog2}
\alpha + \beta = 1, \quad \alpha^{2} + \beta^{-1} = 1.
\end{equation*}
Since we require that $p \equiv 1, 9 \pmod{20}$ we are compelled to ask: how many of these primes have a FPR? We can follow the methods used in \cite[\S 8]{LenstraArtin}, and also examine Shanks's discussion in \cite[p.\ 167]{ShanksFPR}. Since we are now only concerned with $p\equiv 1, 9 \pmod{20}$ we find that the asymptotic density should be $\frac{9}{38}A$, where $A = \prod_{p=2}^{\infty} \left( 1- \frac{1}{p(p-1)}\right) \approx 0.3739558138$ is  Artin's constant.  This leads us to
\begin{thm}\label{t2}
Let $G(x)$ be the number of primes $p\leq x$ for which $p$ satisfies the $G_{4}$ construction. Then, under the Extended Riemann Hypothesis
\begin{equation*}\label{t2:e1}
G(x) \sim \frac{9}{38} \pi(x) \prod_{p=2}^{\infty} \left( 1- \frac{1}{p(p-1)}\right) \sim (0.08856\ldots) \pi(x).
\end{equation*}
\end{thm}
\section{Conclusion}
One can show that, for $p>7$ there can be no primitive root $\alpha$ modulo $p$ that satisfies $\alpha + \alpha^{-1} \equiv 1 \pmod p$. (Suppose there were: then $\alpha^{2} + 1 \equiv \alpha \pmod{p}$ so that $\alpha^{3}+ \alpha^{2} + 1 \equiv \alpha^{2} \pmod{p}$ whence $\alpha^{3} \equiv -1\pmod{p}$. Hence $\alpha^{6} \equiv 1\pmod{p}$ --- a contradiction for $p>7$.) From this, it follows that $x^{p-2} + x - 1$ is never primitive over $\mathbb{F}_p$ for $p>7$.


Consider the following question: given $1\leq i\leq j \leq p-2$, let $d(i, j)$ denote the density of primes for which there is a primitive root $\alpha$ satisfying $\alpha^{i} + \alpha^{j} \equiv 1 \pmod p$. The above comments show that $d(1, p-2) = 0$; Theorem \ref{t1} shows that under ERH, $d(1, 2) \approx 0.2657$.  What can be said about $d(i, j)$ for other prescribed pairs $(i, j)$?  In the case $i=j$, we have $2 \alpha^i \equiv  1 \pmod{p}$ and thus $\alpha^i = \frac{p-1}{2}$. In particular, if $(i, p-1) =1$ then it is equivalent to ask for the density of primes such that $\frac{p-1}{2}$ is a primitive root modulo $p$.  We have not been able to find a reference for this in the literature, though computational evidence seems to suggest that this value should be close to Artin's constant $0.37395\ldots$.

When $i \neq j$, it is easy to see that  $d(2, \frac{p-1}{2} + 1) = d(1,2)$.  Therefore, under ERH the trinomial $x^{\frac{p-1}{2}+1} + x^2 - 1$ is primitive over $\mathbb{F}_p$ for infinitely many primes $p$. 
More generally, we can show that for $p > 3i$ there does not exist a primitive root $\alpha$ such that $\alpha^{\frac{p-1}{2} +i} + \alpha^{\frac{p-1}{2} +2i} \equiv 1 \pmod{p}$, and thus $d(\frac{p-1}{2} +i, \frac{p-1}{2} + 2i) =0$.  Similarly, $d(i, 2i+\frac{p-1}{2}) =0$.  Indeed, if $\alpha^i - \alpha^{2i} \equiv 1 \pmod{p}$ for a primitive $\alpha$, we obtain $\alpha^{3i}\equiv  \alpha^{2i} - \alpha^i \equiv -1\pmod{p}$. Hence we can show that if $p > 6i$ there is no primitive element $\alpha$ such that $\alpha^i + \alpha^{2i+\frac{p-1}{2}} \equiv 1 \pmod{p}$. Using the same arguments as before,  we can also show that $d(i, p-1-i) =0$ for any prefixed $i$.


\begin{thebibliography}{18}




\bibitem{chae2004}
S. C. Chae and Y. O. Park, \emph{Low complexity encoding of improved
  regular {LDPC} codes}, 2004 IEEE 60th Vehicular Technology Conference (VTC2004-Fall, Los Angeles, CA, September 26-29, 2004),  vol.~4, 2004, ~2535--2539.


\bibitem{MR1209243}
S. D. Cohen and G. L. Mullen, \emph{Primitive elements in finite fields
  and {C}ostas arrays}, Appl. Algebra Engrg. Comm. Comput. \textbf{2} (1991),
  no.~1, 45--53.


\bibitem{COT1}
S.~D. Cohen, T.~Oliveira~e Silva, and T.~S. Trudgian.
\emph{A proof of the conjecture of {C}ohen and {M}ullen on sums of
  primitive roots}. Math. Comp., to appear.



\bibitem{costas1975}
J.~P. Costas, \emph{Medium constraints on sonar design and performance}, Proceedings of EASCON (Washington, D.C., September 29-October 1, 1975),  ~68A--68L.




\bibitem{drakakis}
K. Drakakis, \emph{A review of {C}ostas arrays}, J. Appl. Math.
  \textbf{2006} (2006), 1--32.


\bibitem{MR749508}
S. W. Golomb, \emph{Algebraic constructions for {C}ostas arrays}, J.
  Combin. Theory Ser. A \textbf{37} (1984), no.~1, 13--21.


\bibitem{golomb1992}
S.~W. Golomb, \emph{The {$T_4$} and {$G_4$} constructions for {C}ostas arrays}, IEEE
  Trans. Inform. Theory \textbf{38} (1992), no.~4, 1404--1406.


\bibitem{GolombGong}
S.~W. Golomb and G.~Gong.
\emph{The status of {C}ostas arrays}, IEEE Trans. Inform. Theory, \textbf{53} (2007), no. 11, 4260--4265.


\bibitem{golomb1984}
S.~W. Golomb and H. Taylor,  \emph{Constructions and properties of {C}ostas arrays}, Proc. IEEE
  \textbf{72} (1984), no.~9, 1143--1163.

\bibitem{MR674209}
S.~W. Golomb and H. Taylor, \emph{Two-dimensional synchronization
  patterns for minimum ambiguity}, IEEE Trans. Inform. Theory \textbf{28}
  (1982), no.~4, 600--604.




\bibitem{levanon}
N. Levanon and E. Mozeson, \emph{Radar signals}, John Wiley \& Sons, 2004.

\bibitem{LenstraArtin}
H. W. Lenstra,  \emph{On Artin's conjecture and Euclid's algorithm in global fields}, Inventiones math. \textbf{42} (1977), 201-224.

\bibitem{Phong}
B. M. Phong, \emph{Lucas Primitive Roots}, Fibonacci Quart. \textbf{29} (1991), no. 1, 66-71.


\bibitem{maric1995}
S.~V. Maric, M.~D. Hahm, and E.~L. Titlebaum, \emph{Construction and
  performance analysis of a new family of optical orthogonal codes for {CDMA}
  fiber-optic networks}, IEEE Trans. Commun. \textbf{43} (1995), no.~234,
  485--489.


\bibitem{moreno1990}
O. Moreno and J. Sotero, \emph{Computational approach to {C}onjecture
  {A} of {G}olomb}, Congr. Numer. \textbf{70} (1990), 7--16.


\bibitem{Sander}
J.~W. Sander.
\newblock \emph{On {F}ibonacci primitive roots},
\newblock Fibonacci Quart. \textbf{28} (1990), no. 1, 79--80.

\bibitem{ShanksFPR}
D.~Shanks.
\newblock \emph{Fibonacci primitive roots},
\newblock Fibonacci Quart. \textbf{10} (1972), no. 2, 163--181.

\bibitem{ShanksTaylor}
D.~Shanks and L.~Taylor.
\newblock \emph{An observation on {F}ibonacci primitive roots},
\newblock Fibonacci Quart. \textbf{11} (1973), no. 2, 159--160.











\end{thebibliography}
\end{document}